\documentclass[12pt,onecoulme]{article}
\usepackage{amsfonts}
\usepackage{mathrsfs}
\usepackage{amssymb}
\usepackage{color, amsmath,amssymb, amsfonts, amstext,amsthm, latexsym}

\usepackage{amssymb, epsfig, amssymb, latexsym}
\usepackage{amsmath}
\usepackage{graphicx}
\usepackage{longtable}

\numberwithin{equation}{section}

\allowdisplaybreaks

\newtheorem{definition}{Definition}[section]
\newtheorem{theorem}{Theorem}[section]
\newtheorem{lemma}{Lemma}[section]

\newcommand{\be}{\begin{eqnarray}}
\newcommand{\ee}{\end{eqnarray}}
\newcommand{\ce}{\begin{eqnarray*}}
\newcommand{\de}{\end{eqnarray*}}
\title{On the convergence of solutions for SPDEs under perturbation of the domain
 \footnote{The author acknowledges the support provided by NSFs of
China (No.11271013) and the Fundamental Research Funds for the
Central Universities, HUST: 2012QN028.} }

\author{Wenya Wang\\
School of Mathematics and Statistics,\\
Huazhong University of Science and Technology,\\ Wuhan, 430074, China\\
E-mail: wenyawanghust@gmail.com\\
Zhongkai Guo\\
School of Mathematics and Statistics,\\
Huazhong University of Science and Technology,\\ Wuhan, 430074, China\\
E-mail: zhongkaiguo@hust.edu.cn\\
Jicheng Liu\\
School of Mathematics and Statistics,\\
Huazhong University of Science and Technology,\\ Wuhan, 430074, China\\
E-mail: \\ }

\begin{document}

\date{}
\maketitle

\begin{abstract}
We concern the effect of domain perturbation on the behaviour of
stochastic partial differential equations subject to the Dirichlet
boundary condition. Under some assumptions, we get an estimate for
the solutions under changes of the domain.
\end{abstract}

%%%%%%%%%%%%%%%%%%%%%%%%%%%%%%%%%%%%%%%%%%
\section{Introduction}
\quad\, Stochastic partial differential equations (SPDEs) have a
broad spectrum of application including natural sciences and
economics. The purpose of this article is to study the behavior of
solution for stochastic partial differential equations with
Dirichlet boundary condition under the singular domain
perturbations, which means that change of variables is not possible
on these domains. Under property conditions, we show how solutions
of stochastic differential equations behave as a sequence of domains
$\Omega_n$ converges to an open set $\Omega$ in a certain sense. The
motivation to study domain perturbation comes from various sources.
The main ones include shape optimization, solution structure of
nonlinear problems and numerical analysis.\par

 Domain
perturbation or sometimes referred to as ``perturbation of the
boundary" for boundary value problems is a special topic in
perturbation problems. The main characteristic is that the operators
and the nonlinear term live in differential spaces which lead to the
solution of differential equation live in differential spaces.
Domain perturbation appears to be a simple problem if we are only
interested in smooth perturbation of the domain. This is because we
could perform a change of variables to consider the perturbed
problems in a fixed domain and only perturb the coefficients. Hence,
it turns back to a standard perturbation problem and we may apply
standard techniques such as the implicit function theorem, the
Liapunov-Schmidt method and the transversality theorem.
Nevertheless, difficulties arrive when we perform a change of
variables and standard tools are not enough(see \cite{D.H}). When a
change of variables is not possible, domain perturbation is even
more challenging.\par
 The fundamental question in domain perturbation is
to look at how solutions behave upon varying domains. In particular
we would like to know when the solutions converge and what the limit
problem is. There have many papers concern on this topic, which main
under the condition of Mosco convergence. For elliptic equations
case see \cite{Daners,JA} and references therein. In \cite{JA} the
author give a sufficient condition on domains which guarantee the
spectrum behaves continuously. The work of \cite{Daners} prove the
converge of solution for elliptic equations subject to Dirichlet
boundary condition. For parabolic and evolution equation, we
recommend \cite{D.Daners,D} and so on. In \cite{D.Daners} the author
concern domain perturbation for non-autonomous parabolic equations
under the assumption of mosco converge. With such a assumption we
have that the condition of mosco converge is equivalent to the
strong convergence of pseudo resolvent operators for Dirivhlet
promble. Under the assumption of mosco converge, the author of
\cite{D.Daners} get the result of convergence of solutions for both
linear and semilinear parabolic initial value problems subject to
Dirichlet boundary boundary condition as well as persistence of
periodic solutions under domain perturbation. There also some other
papers about invariant manifolds under the domain perturbation see
\cite{JE,Varchon,PSN}, which main concern the converge of invariant
manifolds under the perturbation of the domain.\par For Drichlet
problems, the strong convergence of pseudo resolvent operators is
equivalent to Mosco convergence(see \cite{D.Daners}, Theorem 5.2.4
or \cite{Daners}, Theorem 3.3). In this paper, we take the condition
of strong convergence of pseudo resolvent operators relate to the
domain perturbation. Compare with the Mosco condition, it is more
convenient and effective for proving the convergence of solution for
partial differential equations and stochastic differential equations
under perturbation of the domain.
\par
The remainder of this paper is organized as follows:  In Section 2,
we will review some basic properties of infinitesimal generator and
its semigroups, and existence and unique of solution to stochastic
partial differential equation. The result on the converge of
solution for stochastic differential equation under the perturbation
is described in section $3$\,.

%%%%%%%%%%%%%%%%%%%%%%%%%%%%%%%%%%%%%%%%%%%%%%%%
\section{Preliminaries}
Let $H$ be an infinite dimensional separable Hilbert space with norm
$\|\cdot\|$\,.
Let the sectorial operator $A: D(A)\rightarrow H$ be a
self-adjoint positive linear operator with a compact resolvent.
Then the spectrum of $A$ is real. We denote its spectrum by
\ce
\sigma(A)=\{\lambda_{n}\}_{n=1}^{\infty}, \quad
0<c\leq\lambda_{1}\leq\lambda_{2}\leq\cdots\leq\lambda_{n}\leq\cdots,
\de
and an associated orthonormal family of eigenfunctions by $\{\phi_{n}\}_{n=1}^{\infty}$. Since $A$ is a
sectorial operator, $-A$ is the infinitesimal generator of a
analytic semigroup, which is denoted by
\ce
e^{-At}=\frac{1}{2\pi i}\int_{\gamma}(\lambda I+ A)^{-1}e^{\lambda
t}d\lambda ,
\de
where $\gamma$ is a contour in the resolvent set of
$-A$. Since $A$ is a self-adjoint operator, the formula above is
equivalent to
\ce
e^{-At}u=\sum_{n=1}^{\infty}e^{-\lambda_{n}t}(u,\phi_{n})\phi_{n}.
\de
By the definition $e^{-At}$, we can easily get the following estimate
\ce
\|e^{-At}\|_{L(H, H)}\leq e^{-\lambda_{1}t}\leq 1
\de
for $t\geq 0$\,, which implies that $e^{-At}$ is an analytic contraction semigroup.

Consider the nonlinear stochastic partial differential equation
\be\label{e2.1}
\left\{\begin{array}{ll}
du+Audt=f(u)dt+g(u)dw(t)\,,&\quad in \quad D\times(0, T]\\
u=0\,,&\quad on \quad \partial D\times(0, T]\\
 u(0)=u_{0}\,,&\quad in \quad D
\end{array}\right.
\ee
for $t\in[0,T]$\,. Here $u\in H$\,, $A$ is a sectorial operator, which will be discussed later, $W(t)$ is the standard $\mathbb{R}$-valued
Wiener process on a probability space $(\Omega, \mathcal{F}, \mathbb{P})$\,. For the drift coefficients
$f(u): H\rightarrow H$ and diffusion coefficients $g(u):H\rightarrow H$, we adopt the following assumptions throughout this
paper.\\
$(\mathbf{A.1})$\quad There exists a
constant $k_{2}>0$ such that for any $u, v\in H$, $t\in[0,T]$,
\ce
\|f(u)-f(v)\|^{2}+\|g(u)-g(v)\|^{2}\leq k_{2}\|u-v\|^{2}.
\de
Notice that $(\mathbf{A.1})$ implies
there exists a constant $k_{1}> 0$ such that
\ce
\|f(u)\|^{2}+\|g(u)\|^{2}\leq k_{1}(1+\|u\|^{2})
\de
for any $u\in H$, $t\in[0,T]$.

Now we introduce the definition of solution to Eq.(\ref{e2.1}) and the
existence and uniqueness of solution, Both of them are taken from \cite{DZ}.

\begin{definition}[Mild solution]\label{d2.1}
An $H$-valued predictable process $u(t)$ is called a mild solution
of Eq.(\ref{e2.1}) if for any $t\in [0, T]$ \be\label{s1}
u(t)=e^{-At}u_{0}+\int_{0}^{t}e^{-A(t-s)}f(u(s))ds+\int_{0}^{t}e^{-A(t-s)}g(u(s))dw(s)
\ee
\end{definition}

Let $X_{T}$ denote the set of all continuous $\mathcal
{F}_{t}$-adapted processes valued in $H$ for $0\leq t\leq T$ such
that $E \sup\limits_{0\leq t\leq T}\|u\|^{2}<\infty$. Then $X_{T}$
is a Banach space under the norm \ce \|u\|_{T}=E \sup\limits_{0\leq
t\leq T}\|u\|^{2}. \de Define an operator $\Gamma$ in $X_{T}$ as
follows \be\label{e2.2} \Gamma
u(t)=e^{-At}u_{0}+\int_{0}^{t}e^{-A(t-s)}f(u(s))ds+\int_{0}^{t}e^{-A(t-s)}g(u(s))dw(s),
\ee for $u\in X_{T}$. It is easy to prove that the operator $\Gamma$
is well defined and Lipschitz continuous in $X_{T}$. Then by the
contraction mapping principle, it is easy to prove the existence and
unique of mild solution for the Eq.(\ref{e2.1}) is the following

\begin{theorem}\label{t2.4}
Suppose the condition $(A.1)$ holds true, and $u_{0}$ be
a $\mathcal {F}_{0}$-measurable random field such that
$E\|u_{0}\|^{2}<\infty$. Then the initial-boundary value problem for
the Eq.(\ref{e2.1}) has a unique mild solution $u(t)$ which is a
continuous adapted process in $H$ such that $u\in L^{2}(\Omega;
C([0, T]; H))$ and $$E \sup\limits_{0\leq t\leq
T}\|u\|^{2}\leq C(1+E\|u_{0}\|^{2})$$ for some constant $C>0$.
\end{theorem}

\section{Solution under perturbation of the domain}
In this section, we consider the following perturbation equation of Eq.(\ref{e2.1})
\be\label{e3.2}
\left\{\begin{array}{ll}
du^{\epsilon}+A_{\epsilon}u^{\epsilon}dt=f^{\epsilon}(u^{\epsilon})dt+g^{\epsilon}(u^{\epsilon})dw(t)\,,&\quad in \quad D^{\epsilon}\times(0, T],\\
u^{\epsilon}=0\,,&\quad on \quad \partial D^{\epsilon}\times(0, T],\\
 u^{\epsilon}(0)=u^{\epsilon}_{0}\,,&\quad in \quad D^{\epsilon}
\end{array}\right.
\ee for $\epsilon>0$, where $A_{\epsilon}: D(A_{\epsilon})\subset
H^{\epsilon}\rightarrow H^{\epsilon}$  is a self-adjoint positive
linear operator on a Hilbert space $H^{\epsilon}$ with norm
$\|\cdot\|_{\epsilon}$, and $u_{0}^{\epsilon}$ be a $\mathcal
{F}_{0}$-measurable random field such that
$E\|u_{0}^{\epsilon}\|^{2}<\infty$. We also assume that the
nonlinear terms $f^{\epsilon}:H^{\epsilon}\rightarrow H^{\epsilon}$
and $g^{\epsilon}:H^{\epsilon}\rightarrow H^{\epsilon}$ satisfy
(A.1), which guarantees the existence and unique of mild solution to
the Eq.(\ref{e3.2}). By Theorem \ref{t2.4}, for each ${\epsilon}>0$,
there is an $H$-valued continuous $\mathcal {F}_{t}$-adapted process
$u^{\epsilon}(t)$ such that \be\label{s2}
u^{\epsilon}(t)=e^{-A_{\epsilon}t}u^{\epsilon}_{0}+\int_{0}^{t}e^{-A_{\epsilon}(t-s)}f(u^{\epsilon}(s))ds+\int_{0}^{t}e^{-A_{\epsilon}(t-s)}g(u^{\epsilon}(s))dw(s)
\ee for any $t\in [0,T]$ and $u^{\epsilon}\in L^{2}(\Omega; C([0,
T]; H^{\epsilon}))$.

Note that the solutions value in different function spaces $H^{\epsilon}$ for different ${\epsilon}$.
To deal with domain perturbation,
we assume there exist bound linear operators $\mathbf{P}$ and $\mathbf{Q}$ such that
\ce
\mathbf{P}:H\rightarrow H^{\epsilon}, \quad \mathbf{Q}:H^{\epsilon}\rightarrow H\,,\quad Q\circ
P=I,
\de
\ce
\|\mathbf{P}\|_{\mathcal{L}(H, H^{\epsilon})}\leq 2, \quad \|\mathbf{Q}\|_{\mathcal{L}( H^{\epsilon}, H)}\leq 2,
\de
and
\ce
\|\mathbf{P}u\|_{H^{\epsilon}}\rightarrow
\|u\|_{H}, \quad\mbox{as}~~ \epsilon\rightarrow 0
\de
for all $u\in H$.

To derive the solution of Eq.(\ref{e3.2}) converges to the solution of Eq.(\ref{e2.1}), we also impose the following hypotheses\\
$(\mathbf{H.1})$ For $A$ and $A_{\epsilon}$, we assume
\ce
\|A_{\epsilon}^{-1}\mathbf{P}-\mathbf{P}A^{-1}\|_{\mathcal{L}(H, H^{\epsilon})}=\tau(\epsilon) \rightarrow 0 \quad as \quad \epsilon\rightarrow 0.
\de
 $(\mathbf{H.2})$ We assume that the nonlinear terms
 $g^{\epsilon}\,,f^{\epsilon}:H^{\epsilon}\rightarrow H$ for $0\leq \epsilon \leq\epsilon_{0}$,
 satisfy:
 \begin{itemize}
 \item $f^{\epsilon}$ and $g^{\epsilon}$ approximate $f$ and $g$
 in the following sense,
\ce
\sup\limits_{u\in H}\|f^{\epsilon}(\mathbf{P}u)-
 \mathbf{P}f(u)\|^{2}_{H^{\epsilon}}=\tau_{1}(\epsilon)\rightarrow 0,\quad
 \mbox{as}~~ \epsilon\rightarrow 0.
\de
\ce
\sup\limits_{u\in H}\|g^{\epsilon}(\mathbf{P}u)-
\mathbf{P}g(u)\|^{2}_{H^{\epsilon}}=\tau_2(\epsilon)\rightarrow
0,\quad \mbox{as}~~ \epsilon\rightarrow 0\,.
\de
 \item $f$ and $f^{\epsilon}$ have the uniformly bounded support, that is
\ce
Supp f\subset D_{R}=\{u\in H: \|u\|_{H}\leq R\}
\de
\ce
Supp f^{\epsilon}\subset D_{R}=\{u^{\epsilon}\in H^{\epsilon}: \|u^{\epsilon}\|_{H^{\epsilon}}\leq R\}
\de
 \end{itemize}
$(\mathbf{H.3})$ For initial value $u_{0}$ and $u_{0}^{\epsilon}$, we assume
\ce
E\|u_{0}^{\epsilon}-\mathbf{P}u_{0}\|^{2}_{H^{\epsilon}}=\tau_{0}(\epsilon)\rightarrow 0,\quad \mbox{as}~~ \epsilon\rightarrow 0.
\de

By the condition $(\mathbf{H.1})$ we have the following result,
which concerns the relationship of spectrum between $A$ and
$A_{\epsilon}$ (see \cite{JE}).
\begin{lemma}\label{l3.1}
If $K_{0}$ is a compact set of the
complex plane with $K_{0}\subset \rho(-A)$, the resolvent set of
$A$, and hypothesis $\mathbf{(H.1)}$ is satisfied, then there exists
$\epsilon_{0}(K_{0})>$ such that $K_{0}\subset\rho(-A_{\epsilon})$
for all $0<\epsilon\leq\epsilon_{0}(K_{0})$. Moreover, we have the
estimates
$$\|(\lambda I-A_{\epsilon})^{-1}\|_{\mathcal{L}(H^{\epsilon}, H^{\epsilon})}\leq C(K_{0})\,$$
for all $\lambda\in K_{0},\, 0<\epsilon\leq\epsilon_{0}(K_{0})$.
\end{lemma}

The result implies the upper semi-continuity of the spectrum, that is, if $\lambda_{\epsilon}\in
\sigma(A_{\epsilon})$ and $\lambda_{\epsilon}\rightarrow\lambda$
then $\lambda\in\sigma(A)$. Also we have the resolvent operator
estimate as follows (see \cite{JE}).
\begin{lemma}\label{l3.2}
Let the condition $\mathbf{(H.1)}$ be satisfied, if
$\lambda\in\rho(-A)$ and $\epsilon$ is small enough so that
$\lambda\in\rho(-A_{\epsilon})$, we have
$$\|(\lambda+A_{\epsilon})^{-1}\mathbf{P}-\mathbf{P}(\lambda+A)^{-1}\|_{\mathcal{L}(H,H^{\epsilon})}\leq C(\epsilon,
\lambda)\tau(\epsilon)\rightarrow 0, \quad \mbox{as}~~
\epsilon\rightarrow 0.$$
\end{lemma}
 As we all known, the relationship
between resolvent operator and semigroup is denoted by
\be\label{r2}
e^{-At}=\frac{1}{2\pi i}\int_{\gamma}(\lambda I+ A)^{-1}e^{\lambda
t}d\lambda
\ee
where $\gamma$ is the boundary of
$\Sigma_{-a,\phi}=\{\lambda\in \mathbb{C}:|arg(\lambda+a)|\leq
\pi-\phi\}\subset \rho(-A)$,\,  $\phi\in (0, \frac{\pi}{2})$.
 For simply we take the $a=0,\, \phi=\frac{\pi}{4}$. Then we have $$\gamma=\gamma_{1}\cup\gamma_{2}
 =\{re^{-i\frac{3\pi}{4}}: 0\leq r<\infty\}\cup \{re^{i\frac{3\pi}{4}}: 0\leq r<\infty\} $$
and $C(\epsilon, \lambda)\leq 6$ for all $\lambda\in
\Sigma_{0,\frac{\pi}{4}}$. From Lemma \ref{l3.2} we have the
following estimate.
\begin{lemma}\label{l3.3}
Let $(\mathbf{H.1})$ be satisfied. Then we have
$$\|e^{-A_{\epsilon}t}\mathbf{P}-\mathbf{P}e^{-At}\|_{\mathcal{L}(H, H^{\epsilon})}\leq\frac{C}{r}\tau(\epsilon)
\rightarrow 0,~~\mbox{as}~ \epsilon\rightarrow 0$$ for any $t\in
[r,T]$, here $r>0$.
\end{lemma}
\begin{proof} By (\ref{r2}) and Lemma \ref{l3.2}, we can estimate
\ce
&&\|e^{-A_{\epsilon}t}\mathbf{P}-\mathbf{P}e^{-At}\|_{\mathcal{L}(H,
H^{\epsilon})}\\
&=&\|\frac{1}{2\pi i}\int_{\gamma}(\lambda I+
A_{\epsilon})^{-1}\mathbf{P}e^{\lambda t}d\lambda-\frac{1}{2\pi
i}\int_{\gamma}\mathbf{P}(\lambda I+ A)^{-1}e^{\lambda
t}d\lambda\|_{\mathcal{L}(H,
H^{\epsilon})}\\
&\leq&C |\int_{\gamma_{1}\cup\gamma_{2}}\tau(\epsilon)e^{\lambda
t}d\lambda|.
\de
For $\lambda\in\gamma_{1}\cup\gamma_{2}$, we compute $|e^{\lambda
t}|=|e^{rte^{-i\frac{3\pi}{4}}}|=e^{-\frac{\sqrt{2}}{2}rt}$ with
$0\leq r\leq +\infty$. Then we have
$$\|e^{-A_{\epsilon}t}\mathbf{P}-\mathbf{P}e^{-At}\|_{\mathcal{L}(H,
H^{\epsilon})}\leq
C\tau(\epsilon)\int_{0}^{+\infty}e^{-\frac{\sqrt{2}}{2}rt}dr\leq\frac{C}{t}\tau(\epsilon)\,.$$
Hence by $(\mathbf{H.2})$ \ce
\|e^{-A_{\epsilon}t}\mathbf{P}-\mathbf{P}e^{-At}\|_{\mathcal{L}(H,
H^{\epsilon})}\leq\frac{C}{r}\tau(\epsilon)\rightarrow
0,~~\mbox{as}~~ \epsilon\rightarrow 0 \de for any $t\in[r,T]$, here
$r>0$.
\end{proof}

Now we state and prove our main result as the following.
\begin{theorem}\label{t3.1} Suppose the conditions  $(\mathbf{H.1})$ to
$(\mathbf{H.3})$ hold true. Then we have \ce E\sup\limits_{0\leq
t\leq T}\|u^{\epsilon}(t)-\mathbf{P}u(t)\|_{H^{\epsilon}}^{2}\leq
\frac{C(T,R)(r^{2}+r+\tau_{0}(\epsilon)+\tau_{1}(\epsilon)+\tau(\epsilon))}{1-C(T)k_{2}}.
\de In particular, \ce E\sup\limits_{0\leq t\leq
T}\|u^{\epsilon}(t)-\mathbf{P}u(t)\|_{H^{\epsilon}}^{2}\rightarrow
0, \de when we first let $\epsilon\rightarrow 0$ and then
$r\rightarrow 0$.
\end{theorem}
\begin{proof} From equation (\ref{s1}) and equation (\ref{s2}), we have
\ce
&&E\sup\limits_{0\leq t\leq
T}\|u^{\epsilon}(t)-\mathbf{P}u(t)\|_{H^{\epsilon}}^{2}\\
&=&E\sup\limits_{0\leq t\leq
T}\|e^{-A_{\epsilon}t}u^{\epsilon}_{0}-\mathbf{P}e^{-At}u_{0}
+\int_{0}^{t}e^{-A_{\epsilon}(t-s)}f^{\epsilon}(u^{\epsilon})-\mathbf{P}e^{-A(t-s)}f(u)ds\\
&&+\int_{0}^{t}e^{-A_{\epsilon}(t-s)}g^{\epsilon}(u^{\epsilon})-\mathbf{P}e^{-A(t-s)}g(u)dw(s)\|_{H^{\epsilon}}^{2}\\
&\leq&3E\sup\limits_{0\leq t\leq T}\|e^{-A_{\epsilon}t}u^{\epsilon}_{0}-\mathbf{P}e^{-At}u_{0}\|_{H^{\epsilon}}^{2}\\
&&+3E\sup\limits_{0\leq t\leq T}\|\int_{0}^{t}e^{-A_{\epsilon}(t-s)}f^{\epsilon}(u^{\epsilon})-\mathbf{P}e^{-A(t-s)}f(u)ds\|_{H^{\epsilon}}^{2}\\
&&+3E\sup\limits_{0\leq t\leq T}\|\int_{0}^{t}e^{-A_{\epsilon}(t-s)}g^{\epsilon}(u^{\epsilon})-\mathbf{P}e^{-A(t-s)}g(u)dw(s)\|_{H^{\epsilon}}^{2}\\
&=:&3I_{1}+3I_{2}+3I_{3}
\de
Next we will estimate $I_{1}$, $I_{2}$ and $I_{3}$ respectively. Fix $r$ sufficient small. For $I_{1}$ we
have
\ce
I_{1}&\leq &E\sup\limits_{r\leq t\leq
T}\|e^{-A_{\epsilon}t}u^{\epsilon}_{0}-\mathbf{P}e^{-At}u_{0}\|_{H^{\epsilon}}^{2}
+E\sup\limits_{0\leq t\leq
r}\|e^{-A_{\epsilon}t}u^{\epsilon}_{0}-\mathbf{P}e^{-At}u_{0}\|_{H^{\epsilon}}^{2}
\de
with
\ce
&&E\sup\limits_{r\leq t\leq
T}\|e^{-A_{\epsilon}t}u^{\epsilon}_{0}-\mathbf{P}e^{-At}u_{0}\|_{H^{\epsilon}}^{2}\\
&\leq& 2E\sup\limits_{r\leq t\leq
T}\|e^{-A_{\epsilon}t}u_{0}^{\epsilon}-e^{-A_{\epsilon}t}\mathbf{P}u_{0}\|_{H^{\epsilon}}^{2}
+2E\sup\limits_{r\leq t\leq
T}\|e^{-A_{\epsilon}t}\mathbf{P}u_{0}-\mathbf{P}e^{-At}u_{0}\|_{H^{\epsilon}}^{2}\\
&\leq& CE\sup\limits_{r\leq t\leq
T}\|u_{0}^{\epsilon}-\mathbf{P}u_{0}\|_{H^{\epsilon}}^{2}+C\frac{\tau(\epsilon)}{r}\\
&\leq&C\tau_{0}(\epsilon)+C\frac{\tau(\epsilon)}{r}
 \de and \ce
&&E\sup\limits_{0\leq t\leq
r}\|e^{-A_{\epsilon}t}u^{\epsilon}_{0}-\mathbf{P}e^{-At}u_{0}\|_{H^{\epsilon}}^{2}\\\
&\leq&\!\!\!3E\sup\limits_{0\leq t\leq
r}\|e^{-A_{\epsilon}t}u_{0}^{\epsilon}-u_{0}^{\epsilon}\|_{H^{\epsilon}}^{2}\!\!+3E\|u_{0}^{\epsilon}-\mathbf{P}u_{0}\|_{H^{\epsilon}}^{2}
\!\!+3E\sup\limits_{0\leq t\leq
r}\|\mathbf{P}u_{0}-\mathbf{P}e^{-At}u_{0}\|_{H^{\epsilon}}^{2}\\
&\leq&\!\!\!CE\sup\limits_{0\leq t\leq
r}\|e^{-A_{\epsilon}t}-I\|_{{\mathcal{L}(H^{\epsilon},
H^{\epsilon})}}^{2}+CE\|u_{0}^{\epsilon}-\mathbf{P}u_{0}\|_{H^{\epsilon}}^{2}
+CE\sup\limits_{0\leq t\leq r}\|e^{-At}-I\|_{H}^{2}\\
&\leq&\!\!\!CE\sup\limits_{0\leq t\leq
r}\|e^{-A_{\epsilon}t}-I\|_{{\mathcal{L}(H^{\epsilon},
H^{\epsilon})}}^{2}+C\tau_{0}(\epsilon)+CE\sup\limits_{0\leq t\leq
r}\|e^{-At}-I\|_{{\mathcal{L}(H, H)}}^{2},
\de
where the contraction of $e^{-At}$, Lemma \ref{l3.3} and $(\mathbf{H.3})$ are used.
Therefore
\ce
I_{1}&\leq&C\big(E\sup\limits_{0\leq
t\leq r}\|e^{-A_{\epsilon}t}-I\|_{{\mathcal{L}(H^{\epsilon},
H^{\epsilon})}}^{2}+E\sup\limits_{0\leq t\leq
r}\|e^{-At}-I\|_{{\mathcal{L}(H, H)}}^{2}\big)\\
&&+C(\frac{\tau(\epsilon)}{r}+\tau_{0}(\epsilon)\big). \de For
$I_{2}$ we have \ce I_{2}&\leq &2E\sup\limits_{0\leq t\leq
T}\|\int_{0}^{t}e^{-A_{\epsilon}(t-s)}(f^{\epsilon}(u^{\epsilon})-\mathbf{P}f(u))\|_{H^{\epsilon}}^{2}\\
&&+2E\sup\limits_{0\leq t\leq
T}\|\int_{0}^{t}(e^{-A_{\epsilon}(t-s)}\mathbf{P}-\mathbf{P}e^{-A(t-s)})f(u)ds\|_{H^{\epsilon}}^{2}\\
&\leq & 4E\sup\limits_{0\leq t\leq
T}\|\int_{0}^{t}e^{-A_{\epsilon}(t-s)}(f^{\epsilon}(u^{\epsilon})-f^{\epsilon}(\mathbf{P}u))ds\|_{H^{\epsilon}}^{2}\\
&&+ 4E\sup\limits_{0\leq t\leq
T}\|\int_{0}^{t}e^{-A_{\epsilon}(t-s)}(f^{\epsilon}(\mathbf{P}u)-\mathbf{P}f(u))ds\|_{H^{\epsilon}}^{2}\\
&&+2E\sup\limits_{0\leq t\leq
T}\|\int_{0}^{t}(e^{-A_{\epsilon}(t-s)}\mathbf{P}-\mathbf{P}e^{-A(t-s)})f(u)ds\|_{H^{\epsilon}}^{2}\\
&\leq &4T^{2}k_{2}E\sup\limits_{0\leq t\leq
T}\|u^{\epsilon}-\mathbf{P}u\|_{H^{\epsilon}}^{2}+4T^{2}\|f^{\epsilon}(\mathbf{P}u)-\mathbf{P}f(u)\|_{H^{\epsilon}}^{2}\\
&&+2E\sup\limits_{0\leq t\leq
T}\|\int_{0}^{t}(e^{-A_{\epsilon}(t-s)}\mathbf{P}-\mathbf{P}e^{-A(t-s)})f(u)ds\|_{H^{\epsilon}}^{2}.
\de
Denote $I_{21}=E\sup\limits_{0\leq t\leq
T}\|\int_{0}^{t}(e^{-A_{\epsilon}(t-s)}\mathbf{P}-\mathbf{P}e^{-A(t-s)})f(u)ds\|_{H^{\epsilon}}^{2}$.
Then we have
\ce
I_{21}
&=&E\sup\limits_{0\leq t\leq
T}\|\int_{t-r}^{t}(e^{-A_{\epsilon}(t-s)}\mathbf{P}-\mathbf{P}e^{-A(t-s)})f(u)ds\\
&&+\int_{0}^{t-r}(e^{-A_{\epsilon}(t-s)}\mathbf{P}-\mathbf{P}e^{-A(t-s)})f(u)ds\|_{H^{\epsilon}}^{2}\\
&\leq& 2E\sup\limits_{0\leq t\leq
T}\|\int_{t-r}^{t}(e^{-A_{\epsilon}(t-s)}\mathbf{P}-\mathbf{P}e^{-A(t-s)})f(u)ds\|_{H^{\epsilon}}^{2}\\
&&+2E\sup\limits_{0\leq t\leq
T}\|\int_{0}^{t-r}(e^{A_{\epsilon}(t-s)}\mathbf{P}-\mathbf{P}e^{-A(t-s)})f(u)ds\|_{H^{\epsilon}}^{2}\\
&\leq& C(R,T)r^{2}+C(R,T)\frac{\tau(\epsilon)^2}{r^2}. \de Hence we
obtain \ce I_{2}&\leq&4T^{2}k_{2}E\sup\limits_{0\leq t\leq
T}\|u^{\epsilon}-\mathbf{P}u\|_{H^{\epsilon}}^{2}+4T^{2}\tau_{1}(\epsilon)\\
&&+C(R,T)r^{2}+C(R,T)\frac{\tau(\epsilon)^2}{r^2}\,. \de For $I_{3}$
we have \ce I_{3}&\leq&
CE\int_{0}^{T}\|e^{-A_{\epsilon}(t-s)}g^{\epsilon}(u^{\epsilon})-\mathbf{P}e^{-A(t-s)}g(u)\|_{H^{\epsilon}}^{2}ds\\
&\leq&
CE\int_{0}^{T}\|e^{-A_{\epsilon}(t-s)}(g^{\epsilon}(u^{\epsilon})-\mathbf{P}g(u))\|_{H^{\epsilon}}^{2}ds\\
&&+CE\int_{0}^{T}\|(e^{-A_{\epsilon}(t-s)}\mathbf{P}-\mathbf{P}e^{-A(t-s)})g(u)\|_{H^{\epsilon}}^{2}ds\\
&\leq &
CE\int_{0}^{T}\|e^{-A_{\epsilon}(t-s)}(g^{\epsilon}(u^{\epsilon})-g^{\epsilon}(\mathbf{P}u))\|_{H^{\epsilon}}^{2}ds\\
&&+
CE\int_{0}^{T}\|e^{-A_{\epsilon}(t-s)}(g^{\epsilon}(\mathbf{P}u)-\mathbf{P}g(u))\|_{H^{\epsilon}}^{2}ds\\
&&+CE\int_{0}^{T}\|(e^{-A_{\epsilon}(t-s)}\mathbf{P}-\mathbf{P}e^{-A(t-s)})g(u)\|_{H^{\epsilon}}^{2}ds\\
&\leq &CTk_{2}E\sup\limits_{0\leq t\leq
T}\|u^{\epsilon}-\mathbf{P}u\|_{H^{\epsilon}}^{2}+CT\tau_{2}(\epsilon)+CI_{31},
\de
where
\ce
I_{31}=E\int_{0}^{T}\|(e^{-A_{\epsilon}(t-s)}\mathbf{P}-\mathbf{P}e^{-A(t-s)})g(u)\|_{H^{\epsilon}}^{2}ds.
\de
Let $l=t-s$. Note that $t\geq s$ and $0\leq t\leq T$. Then we have
\ce
I_{31}&=&E\int_{0}^{t}\|(e^{-A_{\epsilon}l}\mathbf{P}-\mathbf{P}e^{-Al})g(u)\|_{H^{\epsilon}}^{2}dl\\
&=&E\int_{0}^{r}\|(e^{-A_{\epsilon}l}\mathbf{P}-\mathbf{P}e^{-Al})g(u)\|_{H^{\epsilon}}^{2}dl+
E\int_{r}^{t}\|(e^{-A_{\epsilon}l}\mathbf{P}-\mathbf{P}e^{-Al})g(u)\|_{H^{\epsilon}}^{2}dl\\
&\leq& C(T,R)r+C(T,R)\frac{\tau(\epsilon)}{r}. \de A combination of
$I_1$, $I_2$ and $I_3$, we finally get \ce &&E\sup\limits_{0\leq
t\leq
T}\|u^{\epsilon}(t)-\mathbf{P}u(t)\|_{H^{\epsilon}}^{2}\\
&\leq& C(T)k_{2}E\sup\limits_{0\leq t\leq
T}\|u^{\epsilon}(t)-\mathbf{P}u(t)\|_{H^{\epsilon}}^{2}\\
&+&C(T,R)(r^{2}+r+\tau_{0}(\epsilon)+\tau_{1}(\epsilon))+\frac{\tau(\epsilon)^2}{r^2}+\frac{\tau(\epsilon)}{r}).
\de We can choose a sufficiently small $T$ such that $C(T)k_{2}<1$,
thus \ce &&E\sup\limits_{0\leq t\leq
T}\|u^{\epsilon}(t)-\mathbf{P}u(t)\|_{H^{\epsilon}}^{2}\\&&\leq
\frac{C(T,R)(r^{2}+r+\tau_{0}(\epsilon)+\tau_{1}(\epsilon)+\frac{\tau(\epsilon)^2}{r^2}+\frac{\tau(\epsilon)}{r})}{1-C(T)k_{2}}.
\de In particular, \ce E\sup\limits_{0\leq t\leq
T}\|u^{\epsilon}(t)-\mathbf{P}u(t)\|_{H^{\epsilon}}^{2}\rightarrow 0
\de as $\epsilon\rightarrow 0$ and then $r\rightarrow 0$.
\end{proof}

%++++++++++++++++++++++++++++++++++++++++++++++++++++++++++++++++++++++++++++++++++++++++++++++++++++++++++++++
%++++++++++++++++++++++++++++++++++++++++++++++++++++++++++++++++++++++++++++++++++++++++++++++++++++++++++++++
%\section*{Acknowledgments} The authors would like to thank
%anonymous reviewers for useful comments.

%++++++++++++++++++++++++++++++++++++++++++++++++++++++++++++++++++++++++++++++++++++++++++++++++++++++++++++++
%++++++++++++++++++++++++++++++++++++++++++++++++++++++++++++++++++++++++++++++++++++++++++++++++++++++++++++++

\end{document}